\documentclass[12pt]{amsart}

\usepackage{hyperref}
\hypersetup{
    colorlinks=true,         
    linkcolor=blue,          
    citecolor=blue,          
    filecolor=blue,          
    urlcolor=blue            
}
\usepackage{latexsym,graphicx}

\usepackage[usenames,dvipsnames]{color}

\usepackage[leqno]{amsmath}
\usepackage{amssymb,amsthm}
\usepackage{colordvi,vmargin}
\usepackage[usenames,dvipsnames]{pstricks}

\usepackage{ulem}    

\usepackage{todonotes}
\setpapersize{USletter}
\setmargrb{2.5cm}{2.5cm}{2.5cm}{2.5cm} 

\newtheorem{theorem}{Theorem}[section]
\newtheorem{lemma}[theorem]{Lemma}
\newtheorem{proposition}[theorem]{Proposition}

\theoremstyle{definition}

\theoremstyle{remark}

\newtheorem{example}[theorem]{Example}

\newcounter{FNC}[page]
\def\fauxfootnote#1{{\addtocounter{FNC}{2}$^\fnsymbol{FNC}$%
     \let\thefootnote\relax\footnotetext{\Magenta{$^\fnsymbol{FNC}$#1}}}}
\def\cprime{$'$}

\newcommand{\calA}{{\mathcal A}}
\newcommand{\calB}{{\mathcal B}}
\newcommand{\calF}{{\mathcal F}}
\newcommand{\calG}{{\mathcal G}}
\newcommand{\calH}{{\mathcal H}}
\newcommand{\calS}{{\mathcal S}}

\newcommand{\C}{{\mathbb C}}
\newcommand{\N}{{\mathbb N}}
\newcommand{\R}{{\mathbb R}}
\newcommand{\Z}{{\mathbb Z}}

\newcommand{\ba}{{\bf a}}
\newcommand{\bb}{{\bf b}}
\newcommand{\bc}{{\bf c}}
\newcommand{\bd}{{\bf d}}
\newcommand{\fb}{{\bf f}}
\newcommand{\bg}{{\bf g}}

\newcommand{\taut}{\mbox{\rm taut}}

\DeclareMathOperator{\conv}{conv}

\DeclareMathOperator{\Aff}{Aff}
\DeclareMathOperator{\Log}{Log}
\DeclareMathOperator{\Exp}{Exp}

\newcommand{\simplex}{\includegraphics{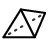}}
\newcommand{\bsimplex}{\includegraphics{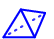}}

\newcommand{\defcolor}[1]{\Blue{#1}}
\newcommand{\demph}[1]{\defcolor{{\sl #1}}}

\title{Degenerations of Real Irrational Toric Varieties}

\author{Elisa Postinghel}
\address{Elisa Postinghel\\
         KU Leuven, Department of Mathematics\\
         Celestijnenlaan 200B, B-3001 Leuven, Belgium}
\email{elisa.postinghel@wis.kuleuven.be}
\urladdr{www.kuleuven.be/wieiswie/c/en/person/00092710}

\author{Frank Sottile}
\address{Frank Sottile\\
         Department of Mathematics\\
         Texas A\&M University\\
         College Station\\
         Texas \ 77843\\
         USA}
\email{sottile@math.tamu.edu}
\urladdr{www.math.tamu.edu/\~{}sottile}

\author{Nelly Villamizar}
\address{Nelly Villamizar\\
         RICAM, Austrian Academy of Sciences\\
         Linz, Austria}
\email{nelly.villamizar@oeaw.ac.at}
\urladdr{https://people.ricam.oeaw.ac.at/n.villamizar/}

\thanks{This material is based on work done while Postinghel and Sottile were in
  residence at the Institut Mittag-Leffler in Winter 2011.}
\thanks{Research of Sottile is supported in part by NSF grant DMS-1001615.}
\thanks{This material is based upon work supported by the National Science 
Foundation under Grant No. 0932078 000, while Sottile was in 
residence at the Mathematical Science Research Institute (MSRI) in 
Berkeley, California, during the winter semester of 2013.}
\thanks{Research of Postinghel and Villamizar was partially supported by
Marie-Curie IT Network SAGA,
grant PITN-GA-2008-214584.}
\thanks{Research of Postinghel was partially supported by the project 
``Secant varieties, computational complexity and toric degenerations''
realised within the Homing Plus programme of Foundation for
Polish Science, cofinanced from EU, Regional Development Fund.}

\begin{document}

\begin{abstract} 
A real irrational toric variety $X$ is an analytic subset of the simplex associated to a
   finite configuration of real vectors.
 The positive torus acts on $X$ by translation, and we consider limits of sequences of
 these translations.
 Our main result identifies all possible Hausdorff limits of translations of $X$
 as toric degenerations using elementary methods and the geometry of the secondary
   fan of the vector configuration. 
 This generalizes work of Garc\'ia-Puente et al., who used algebraic geometry and 
  work of Kapranov, Sturmfels, and Zelevinsky, when the vectors were integral.
\end{abstract}

\maketitle

\centerline{Dedicated to the memory of Andrei Zelevinsky}

\section*{Introduction}
A not necessarily normal complex projective toric variety $Y_\calA$ is parametrized by
monomials whose exponent vectors form a finite set $\calA$ of integer vectors.
The theory of toric varieties~\cite{CLS} elucidates many ways how the structure of
$Y_\calA$ is encoded in the point set $\calA$.
For example, the set $X_\calA$ of nonnegative real points of $Y_\calA$ is
homeomorphic to the convex hull $\Delta_\calA$ of $\calA$ through a linear projection.

If we drop algebraicity, we may associate a real irrational toric
variety $X_\calA$ to any finite set $\calA$ of real vectors.
This is the analytic subvariety of the standard $\calA$-simplex $\simplex^\calA$
parametrized by monomials with exponent vectors from $\calA$, and it is
homeomorphic to the convex hull $\Delta_\calA$ of $\calA$ 
through a linear projection, by Birch's Theorem from algebraic
 statistics~\cite[p.~168]{Agr90}.
Other aspects of the dictionary between toric varieties $Y_\calA$ and 
sets of integer vectors $\calA$ extend to real irrational toric varieties
$X_\calA$ and finite sets of real vectors $\calA$~\cite{CGS}.

We further extend this dictionary.
When $\calA\subset\Z^d$, the torus $(\C^\times)^\calA$ of the ambient 
projective space of the toric variety $Y_\calA$ acts on it via translations.
Kapranov, Sturmfels, and Zelevinsky~\cite{KSZ1,KSZ2} identified 
the closure in the Hilbert scheme of the set of torus translations
$\{w.Y_\calA\mid w\in(\C^\times)^\calA\}$ as the toric variety associated to the secondary
fan of the set $\calA$.
The cones of this fan correspond to regular subdivisions of $\calA$, which are described
in~\cite[Ch.~7]{GKZ}. 

A consequence is that the only limiting schemes of torus translations of
$Y_\calA$ are one-parameter toric degenerations.
Restricting to the nonnegative part $X_\calA$ of $Y_\calA$ and to positive translations
$w.X_\calA$ for $w\in \R_>^\calA$ gives an identification between limiting positions of 
positive torus translations of $X_\calA$ and toric degenerations of $X_\calA$ by cosets of
one-parameter subgroups of $\R_>^\calA$.
Here, limiting position is measured with respect to the Hausdorff metric on subsets of the
nonnegative part of the projective space.
In~\cite{GSZ} this was used to identify all Hausdorff limits of B\'ezier
patches in geometric modeling.

We extend this identification between Hausdorff limits of torus translates
of $X_\calA$ and toric degenerations from the case
when $\calA$ consists of integer vectors (and thus the methods of~\cite{KSZ1,KSZ2} using
algebraic geometry apply) to the case when $\calA$ consists of any finite set of real vectors,
so that methods from algebraic geometry do not apply.

We do this through a direct argument, identifying all Hausdorff limits of torus
translations of $X_\calA$ with toric degenerations of $X_\calA$.
We state our main theorem.\medskip

\noindent{\bf Theorem~\ref{Th:creatingLimits}.}
{\it Let $\calA\subset\R^d$ be a finite set of real vectors and $\{w_i\mid i\in\N\}\subset\R^\calA_>$ a sequence
in the positive torus. 
  Then there exists a subsequence $\{u_j\mid j\in\N\}\subset\{w_i\mid i\in\N\}$ and a toric
 degeneration $X(\calS,w)$ such that
\[
    \lim_{j\to\infty} u_j.X_\calA\ =\ X(\calS,w)
\]
in the Hausdorff topology on subsets of the simplex $\simplex^\calA$.}\medskip

In particular, if a sequence of torus translates of $X_\calA$ has a limit in the Hausdorff
topology, then that limit is a toric degeneration.
This shows that the space of torus translates of $X_\calA$ is compactified by adding the
toric degenerations. 

Real irrational toric varieties arise naturally in mathematics and its applications.
They are log-linear (toric) models in algebraic statistics~\cite{PRW}, they were critical in the proof of the
improved bound~\cite{BS07} for positive solutions to fewnomial systems, and they are a natural framework in
which to understand properties of B\'ezier patches~\cite{GS,Kr02}.
Toric degenerations have played a role in these last two areas.
In geometric modeling, they are the possible limiting positions of a B\'ezier patch with varying weights,
giving rise to a meaningful generalization of the control polygon~\cite{CGS,GSZ}.
Toric degenerations were a key tool to compute lower
bounds on the number of real solutions~\cite{SS} to systems of polynomial equations.
We believe that further properties of real irrational toric varieties will be useful in all these application
domains.

This paper is organized as follows.
In Section 1, we develop some technical results about sequences in cones and of regular
subdivisions and the secondary fan.
In Section 2, we define a real irrational toric variety $X_\calA$ associated to a configuration
of points $\calA\subset\R^d$, recalling some of its properties and identifying its
torus translations, as well as recalling the Hausdorff metric and topology.
In Section 3, we study toric degenerations of $X_\calA$, relating them to the secondary
fan, and (re)state our main theorem.
Its proof occupies Section 4.

\section{Some geometric combinatorics} 

We develop some
results about sequences in cones, regular subdivisions, and the secondary fan.
Write $\N=\{1,2,\dotsc\}$ for the positive integers, $\R$ for the real numbers, $\R_\geq$
for the nonnegative real numbers, and $\R_>$ for the strictly positive real numbers.
We use the standard notions of polyhedron, polytope, cone, face, etc. from geometric
combinatorics, which may be found in~\cite{DLRS,GKZ,Zi}.
For example, a polyhedron is an intersection of finitely many closed half-spaces, a
polytope is a bounded polyhedron and it is also the convex hull of a finite set of
points. 
Faces of a polyhedron are its intersections with hyperplanes bounding half-spaces
containing it.
The smallest affine space containing a polyhedron is its \demph{affine span} and its
\demph{relative interior} is its interior as a subset of its affine span.

A \demph{cone} is a polyhedron given by half-spaces $\{x\in\R^n\mid\psi(x)\geq 0\}$, where $\psi$ is a linear form.
The boundary hyperplane of such a half-space is the linear subspace 
$\defcolor{\psi^\perp}:=\{x\in\R^n\mid\psi(x)=0\}$.
The intersection of its boundary hyperplanes is the \demph{lineality space} of a cone and
it is \demph{pointed} if its lineality space is the origin.
Two faces $\rho,\sigma$ of a pointed cone $\tau$ are \demph{adjacent} if $\rho\cap\sigma$
is not the vertex of $\tau$.
The affine span of a cone $\sigma$ is its linear span, 
\defcolor{$\langle\sigma\rangle$}. 

A \demph{polyhedral complex} is a finite collection $\Pi$ of polyhedra 
that is closed under taking faces such that the intersection of any two polyhedra in
$\Pi$ is a face of both (or empty).
A polyhedral complex is a \demph{triangulation} if it consists of simplices.
A \demph{fan} is a polyhedral complex consisting of cones.

\subsection{Sequences in cones}\label{S:SeqsInCones}
We formulate two technical results about sequences in cones, Lemmas~\ref{L:finalSubdivision}
and~\ref{L:cone_bounded}, that are used in an essential way in the formulation and proof of  
Theorem~\ref{Th:creatingLimits}.
Let \defcolor{$|x|$} be the usual Euclidean length of $x$ in a real vector space.

A sequence $\{v_i\mid i\in\N\}$, or simply 
\defcolor{$\{v_i\}$},  is \demph{divergent} if it has no bounded
subsequence, that is, if for all $M>0$, there is an $N$ such that if $i>N$, then
$|v_i|>M$.

\begin{lemma}\label{L:far_apart}
 If $\rho$ and $\sigma$ are non-adjacent faces of a pointed cone $\tau$ and 
 $\{r_i\mid i\in\N\}\subset\rho$ and $\{s_i\mid i\in\N\}\subset\sigma$ are 
 divergent sequences, then $\{r_i{-}s_i\mid i\in\N\}$  is divergent. 
\end{lemma}

\begin{proof}
 Let $S$ be the unit sphere centred at the origin. 
 As $\sigma,\rho$ are not adjacent, there is a positive lower bound, $\delta$, for the
 distance $|r{-}s|$ between any pair of points $r\in\rho\cap S$ and $s\in\sigma\cap S$.

 Let $M>0$.  Since $\{r_i\mid i\in\N\}$ and $\{s_i\mid i\in\N\}$ are divergent sequences, 
 there exists $N$ such that  $|r_i|,|s_i|>M/\delta$ for all $i>N$.
 Take $i>N$, and suppose that $|r_i|\geq|s_i|$ 
 (otherwise interchange the sequences). 
 Then
 \begin{equation}
   \label{Eq:stringIneqs}
    |r_i-s_i|\ =\ |s_i|\left|\frac{r_i}{|s_i|}-\frac{s_i}{|s_i|}\right|
     \ \geq\ |s_i| \left|\frac{r_i}{|r_i|}-\frac{s_i}{|s_i|}\right|
     \ >\ \frac{M}{\delta} \delta\ =\ M\,,
 \end{equation}
 which completes the proof.
\end{proof}

The first inequality in~\eqref{Eq:stringIneqs} is elementary geometry:
If $u,v$ are unit vectors and $t\geq 1$, then
\[
    |tu-v|\ \geq\ |u-v|\,.
\]
This is clear if $u=v$.
Otherwise, let $\theta$ be the angle between $u$ and $v$.
Then in the triangle with vertices $tu,u,v$, the angle at $u$ is
$\frac{\pi}{2}+\frac{\theta}{2}$, which is obtuse.
Then this inequality is just that the longest side of a triangle is opposite to its largest 
angle. 

Let $\tau$ be a cone in $\R^n$, $\{v_i\mid i\in\N\}$ a sequence in $\tau$, and $\sigma$
a face of $\tau$. 
We say that $\{v_i\}$ is \demph{$\sigma$-bounded}  if for every linear
function $\psi$ vanishing on $\sigma$, the set $\{\psi(v_i)\mid i\in\N\}$ is bounded in $\R$.
It is immediate that if $\{v_i\}$ 
is both $\sigma$- and $\rho$-bounded, then it is $(\sigma\cap\rho)$-bounded.
The \demph{minimum face of boundedness} of $\{v_i\}\subset\tau$ 
is the smallest face $\sigma$ of
$\tau$ such that $\{v_i\}$ is $\sigma$-bounded.

\begin{example}\label{Ex:mmfb}
 Let $\sigma:=\R_>\cdot(-1,-1)$ and $\rho:=\R_>\cdot(0,-1)$ be rays in $\R^2$, and $\tau$ the cone they span.
 Consider the two sequences $\{v_i\}$ and $\{u_i\}$ in the cone $\tau$, defined for $i\geq 1$ by
\[
   v_i\ :=\ \bigl( -i-\tfrac{1}{i}\,,\, -i-1\bigr) 
    \qquad\mbox{and}\qquad
   u_i\ :=\ \bigl( -i+\sqrt{i}\,,\, -i\bigr) \,.
\]
 Neither sequence is $0$- or $\rho$-bounded and both are $\tau$-bounded.
 However, only $\{v_i\}$ is $\sigma$-bounded.
 Note that both sequences have the same asymptotic direction (along $\sigma$),
\[
   \lim_{i\to\infty}\frac{v_i}{|v_i|}\ =\ 
   \lim_{i\to\infty}\frac{u_i}{|u_i|}\ =\  (-1,-1)\,.
\]
 We display the first few terms of the two sequences in the cone $\tau$ below.
\[
  \begin{picture}(140,81)(-1,0)
   \put(0,0){\includegraphics{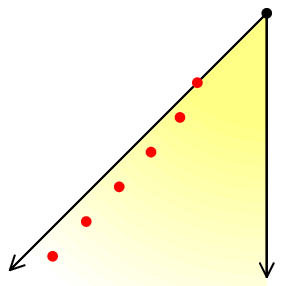}}
   \put(2,15){$\sigma$}  \put(58,13){$\tau$}  \put(86,14){$\rho$}
   \put(22,46){$v_i$}
  \end{picture}
  \qquad
  \begin{picture}(140,81)(-1,0)
   \put(0,0){\includegraphics{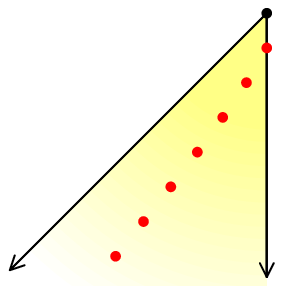}}
   \put(2,15){$\sigma$}  \put(58,13){$\tau$}  \put(86,14){$\rho$}
   \put(22,46){$u_i$}
  \end{picture}
\]
\end{example}

\begin{lemma}\label{L:finalSubdivision}
 Let $\{v_i\mid i\in\N\}\subset\tau$ be a sequence in a cone $\tau$.
 Then there is a face $\sigma$ of $\tau$ and a subsequence $\{u_i\mid i\in\N\}$ of $\{v_i\}$ such that
 $\sigma$ is the minimum face of boundedness of any subsequence of $\{u_i\}$.
\end{lemma}

\begin{proof}
 The set \defcolor{$\calB$} of faces $\sigma$ of $\tau$ for which $\{v_i\}$ 
 has a subsequence that is $\sigma$-bounded is nonempty ($\tau\in\calB$).
 It forms an order ideal, for if $\sigma\in\calB$ and $\sigma\subset\rho$ a face of $\tau$, then 
 $\rho\in\calB$.
 Let \defcolor{$\sigma$} be a minimal element of $\calB$ and $\{u_i\mid i\in\N\}$ be a $\sigma$-bounded subsequence.
 The minimum face of boundedness of a subsequence of $\{u_i\}$ is a subface of $\sigma$ and therefore
equals $\sigma$, by its minimality in $\calB$.
\end{proof}



\begin{lemma}\label{L:cone_bounded}
 Let $\tau$ be a cone in $\R^n$, $\{v_i\mid i\in\N\}\subset\tau$ a sequence whose minimum face of boundedness is
 $\sigma$. 
  Then, for any $v\in\sigma$, all except finitely many elements of the sequence
  $\{v_i-v\mid i\in\N\}$ lie in $\tau$.
\end{lemma}

\begin{proof}
 Suppose by way of contradiction that $v\in\sigma$ and $\{u_j\mid j\in\N\}$ is an
 infinite subset of $\{v_i\}$ such that 
 $\{u_j-v\mid j\in\N\}$ is disjoint from $\tau$. 
 Without loss of generality, assume that $\langle\tau\rangle=\R^n$.
 Let $\Psi$ be a finite irredundant collection of linear forms defining $\tau$,
\[
   \tau\ =\  \{x\in\R^n\,\mid\, \psi(x)\geq0\quad\mbox{for all }\psi\in\Psi\}\,.
\]
 Since $\Psi$ is irredundant and $\langle\tau\rangle=\R^n$, each form $\psi\in\Psi$
 supports a facet $\psi^\perp\cap\tau$ of $\tau$.

 Since $\{u_j-v\mid j\in\N\}\cap\tau=\emptyset$, for each $j\in\N$ there is a form
 $\psi\in\Psi$ with $\psi(u_j-v)<0$ so that $\psi(u_j)<\psi(v)$.
 Since $\Psi$ is finite, there is a subsequence $\{w_k\mid k\in\N\}$ of 
 $\{u_j\}$ and a form $\psi\in\Psi$ such that $\psi(w_k)<\psi(v)$ for all
 $k\in\N$. 
 Since $\{w_k\}\subset\tau$, we have that 
 $\psi(w_k)\in[0,\psi(v))$ for all $k\in\N$.
Thus if $\defcolor{\rho}:=\psi^\perp\cap\tau$ is the face of $\tau$ given by $\psi$,
   then $\{w_k\}$ is $\rho$-bounded.

 Since $\{v_i\}$ has a $\rho$-bounded subsequence,
 we have $\sigma\subset\rho$. 
 Since $v\in\sigma$ and $\rho\subset\psi^\perp$, we have $\psi(v)=0$, which contradicts
 the inequality $0\leq\psi(u_j)<\psi(v)$ that we have for any $j$.
 This concludes the proof.
\end{proof}

\subsection{Regular subdivisions}\label{S:RegSubdv}
Fix a positive integer $d$ and let $\calA\subset\R^d$ be a finite set of points, which we
assume affinely spans $\R^d$.
We use elements of $\calA$ throughout to index coordinates, variables, functions,
etc.
For example, $\R^\calA$ is the space of real-valued functions on $\calA$.
This has a distinguished subspace $\defcolor{\Aff(\calA)}\simeq\R^{d+1}$, consisting of 
functions on $\calA$ that are restrictions of affine functions on $\R^d$.
For $z\in\R^\calA$ and $\ba\in\calA$, we write $z_\ba$ for $z(\ba)$, its $\ba$-th coordinate.

For any subset $\calF\subset\calA$, the extension by zero gives an inclusion 
$\R^\calF\hookrightarrow\R^\calA$ and the restriction of functions given by $w\mapsto w|_\calF$ defines a
map $\R^\calA\to\R^\calF$.

For $\calF\subset\calA$, let \defcolor{$\Delta_\calF$} be the convex hull of
$\calF$ i.e.,
\[
   \defcolor{\Delta_\calF}\ :=\ \Bigl\{\sum_{\fb\in\calF} \mu_{\fb}\fb \mid \mu_\fb\geq0
     \ \mbox{ and }\ 1=\sum_{\fb\in\calF} \mu_{\fb}\Bigr\}\,.
\]
A \demph{polyhedral subdivision $\calS$} of $\calA$ is a collection of subsets $\calF$ of $\calA$,
called \demph{faces} of $\calS$, whose convex hulls 
$\{\Delta_\calF\mid \calF\in\calS\}$ form a polyhedral complex \defcolor{$\Pi_\calS$}
which covers $\Delta_\calA$.
If $\calF,\calG$ are faces of a polyhedral subdivision $\calS$ of $\calA$,
then $\calH:=\calF\cap\calG$ is also a face of $\calS$ and 
$\Delta_\calH=\Delta_\calF\cap\Delta_{\calG}$.
A \demph{facet} is a maximal face of $\calS$, which is a face $\calF$ that affinely spans
$\R^d$ so that $\Delta_\calF$ has dimension $d$.
A \demph{triangulation} of $\calA$ is a polyhedral subdivision $\calS$ in which every
face $\Delta_\calF$ of $\Pi_\calS$ is a simplex with vertices $\calF$.
A polyhedral subdivision $\calS$ is \demph{regular} if there is a 
piecewise-affine concave
function $g$ on $\Delta_\calA$ where the maximal domains on which $g$ is affine
are $\Delta_\calF$ for facets $\calF$ of $\calS$. 
Such a concave function $g$ is \demph{strictly concave} on the subdivision $\calS$.

Elements $\lambda\in\R^\calA$ induce regular subdivisions of $\calA$ in the following way.
Let \defcolor{$P_\lambda$} be the convex hull of the graph of $\lambda$, defined by
\[
    P_\lambda\ :=\ \conv\{ (\ba,\lambda(\ba))\,\mid\, \ba\in\calA\}\,.
\]
Its \demph{upper faces} are those having an outward-pointing normal vector with last coordinate positive.
For an upper face $F$ of $P_\lambda$, let \defcolor{$\calF(F)$} be the points $\ba$ of $\calA$ such that
$(\ba,\lambda(\ba))$ lies on $F$. 
Let \defcolor{$\calS_\lambda$} be the collection of subsets $\calF(F)$ of $\calA$ where
$F$ ranges over the upper faces of $P_\lambda$.
This forms a polyhedral subdivision of $\calA$ as the upper faces of $P_\lambda$ 
form a polyhedral complex whose projection to $\Delta_\calA$ covers
$\Delta_\calA$, and the projection of an upper face $F$ is $\Delta_{\calF(F)}$.

Lastly, $S_\lambda$ is regular---the upper faces of $P_\lambda$ form the
graph of the desired concave function, \defcolor{$g_\lambda$}. 
Conversely, if $\calS$ is a regular subdivision with strictly concave function $g$, then
any $\lambda\in\R^\calA$ satisfying 
$\lambda(\ba)\leq g(\ba)$ with equality if and only if $\ba$ lies in some face of $\calS$
has $\calS=\calS_\lambda$.

\begin{example}
 Let $\calA\subset\R^2$ be a $3\times 3$ grid of nine points.
 Figure~\ref{F:upper_hulls} shows three polyhedral subdivisions of $\calA$ induced
 by elements $\lambda\in\R^\calA$, together with the lifted points
 $\{(\ba,\lambda(\ba)) \mid\ba\in\calA \}$ and the corresponding upper faces.
\begin{figure}[htb]
 \includegraphics[height=90pt]{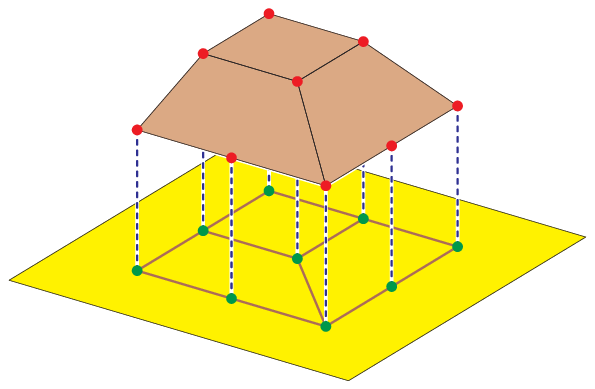}\quad
  \includegraphics[height=90pt]{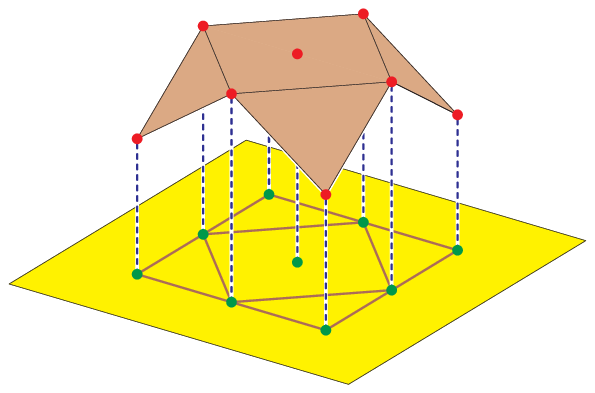}\quad
  \includegraphics[height=90pt]{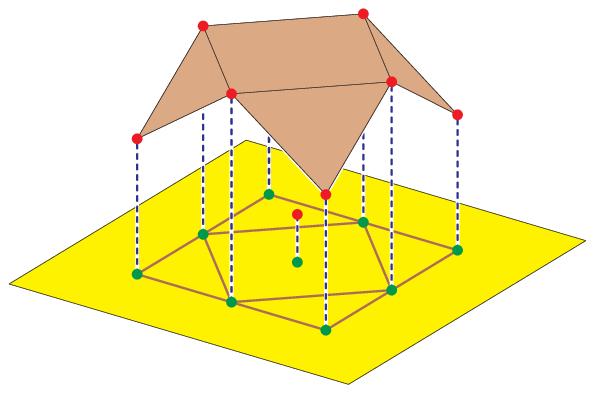}
 \caption{Three regular subdivisions}
 \label{F:upper_hulls}
\end{figure}
 All elements of $\calA$ participate in the first two subdivisions, but the
 center element of $\calA$ does not participate in the third, for 
 it does not lie on an upper face.
\end{example}

A subset $\calF$ of $\calA$ is a subset of a face of $\calS_\lambda$ if and
only if the restriction
$\lambda|_\calF$ of $\lambda$ to $\calF$ is an affine function whose extension to 
$\Delta_\calF$ agrees with the restriction of $g_\lambda$ to $\Delta_\calF$,
where $g_\lambda$ is the strictly convex function whose graph is the upper faces
  of $P_\lambda$.
The minimal subsets of $\calA$ that are not contained in any face of $\calS_\lambda$ are
singletons $\{\bc\}$ that do not participate in the subdivision and doubletons
$\{\ba,\bb\}$ in which both $\ba$ and $\bb$ participate in the subdivision, but no face
contains both, so that the interior of the line segment between the lifted points
  lies below the upper faces.

\begin{lemma}\label{L:convex_combinations}
 Let $\calS=\calS_\lambda$ be a regular subdivision of $\calA$.
 \begin{enumerate}
  \item If $\{\ba,\bb\}\subset\calA$ is not a subset of any face of $\calS$, then there is
    a facet $\calG$ of $\calS$, a point $p\in\Delta_\calG$, and numbers
    $\beta_\ba,\beta_\bb>0$ and $ \alpha_\bg\geq 0$ for $\bg\in\calG$ with  
   \begin{equation}\label{Eq:convex_ii}
    p\ =\ \beta_\ba \ba + \beta_\bb \bb \ =\ \sum_{\bg\in\calG} \alpha_\bg \bg \,.
     \qquad\mbox{where}\quad 
    1\ =\ \beta_\ba+\beta_\bb\ =\ \sum_{\bg\in\calG} \alpha_\bg\ .
   \end{equation}

  \item If $\bc\in\calA$ is not a member of any face of $\calS$, then there is a facet
    $\calG$ of $\calS$ with $\bc\in\Delta_\calG$ and therefore an expression 
   \begin{equation}\label{Eq:convex_i}
    \bc\ =\ \sum_{\bg\in\calG} \alpha_\bg \bg
     \qquad\mbox{where}\quad \alpha_\bg\geq 0 
     \mbox{ and } 1\ =\ \sum_{\bg\in\calG} \alpha_\bg\,.
   \end{equation}

  \item  If $\calG$ is a facet of $\calS$ and $\bd\not\in\calG$, then there is an expression
   \begin{equation}\label{Eq:affine_bd}
     \bd\ =\ \sum_{\bg\in\calG} \alpha_\bg \bg
     \qquad\mbox{where}\quad 1\ =\ \sum_{\bg\in\calG} \alpha_\bg\,,
    \end{equation}
   of $\bd$ as an affine combination of points of $\calG$.

 \end{enumerate}
 In any of {\rm(i)}, {\rm(ii)}, or {\rm(iii)}, if\/ \defcolor{$\widetilde{\lambda}$} is the affine
 function whose restriction to the facet $\calG$ agrees with $\lambda$,
 then  
\[
  \widetilde{\lambda}\biggl(\sum_{\bg\in\calG} \alpha_\bg\bg\biggr)
  \ =\ \sum_{\bg\in\calG} \alpha_\bg \lambda(\bg)\,,
\]
 and we have the (respective) inequalities
  \begin{equation}\label{Eq:liftIneq}
   \beta_\ba\lambda(\ba)+\beta_\bb\lambda(\bb)<\widetilde{\lambda}(p)\,,\ \ 
   \lambda(\bc)<\widetilde{\lambda}(\bc)\,,\ \ 
   \mbox{and}\ \ 
   \lambda(\bd) < \widetilde{\lambda}(\bd)\,.
  \end{equation}
\end{lemma}

\begin{proof}
 If $\{\ba,\bb\}$ is not a subset of any face of $\calS$, then the interior of the
 line segment they span meets the  convex hull $\Delta_\calF$ of
 some facet $\calG$ of $\calS$ in a point $p$.
 This gives the expression~\eqref{Eq:convex_ii}. 
 The first inequality of~\eqref{Eq:liftIneq} expresses that the interior of the segment
 joining the lifted points $(\ba,\lambda(\ba))$ 
 and $(\bb,\lambda(\bb))$ lies strictly below the upper hull of $P_\lambda$.

 If $\bc\in\calA$ is not a member of a face of $\calS$, then there is a facet $\calG$ of
 $\calS$ with $\bc$ lying in $\Delta_\calG$.
 Thus $\bc$ is a convex combination of the points of $\calG$,
 giving~\eqref{Eq:convex_i}. 
 Since $\bc$ lies in no face of $\calS$, the lifted point $(\bc,\lambda(\bc))$
 is below the upper hull of $P_\lambda$, which implies the middle 
 inequality of~\eqref{Eq:liftIneq} as $\widetilde{\lambda}(\bc)$ is the height of the
 point on the upper hull of $P_\lambda$ above $\bc$.

 Finally, as $\calG$ is a facet of $\calS$, its points affinely span $\R^d$, so there is an
 expression of $\bd$ as an affine combination of the points of $\calG$~\eqref{Eq:affine_bd}. 
 The graph of the function $\widetilde{\lambda}$ is the hyperplane supporting the 
 upper facet of the lifted polytope $P_\lambda$ corresponding to $\calG$.
 Then, if $\ba\not\in\calG$, we have $\lambda(\ba)<\widetilde{\lambda}(\ba)$, and
 the third inequality of~\eqref{Eq:liftIneq} is a special case of this.
\end{proof}

\subsection{Secondary fan of a point configuration}

For a regular subdivision $\calS$ of a point configuration $\calA\subset\R^d$, let
$\defcolor{\sigma(\calS)}\subset\R^\calA$ be the (closure of) the set of all functions
$\lambda$ which induce $\calS$. 
This forms a cone in $\R^\calA$ which is full-dimensional if and only if $\calS$ is a
regular triangulation of $\calA$.
The collection of these cones forms the \demph{secondary fan $\Sigma_\calA$} of the point
configuration $\calA$.
Write \defcolor{$\calS_\sigma$} for the subdivision corresponding to a cone $\sigma$ of
the secondary fan. 
The minimal cone of $\Sigma_\calA$ is the linear space $\Aff(\calA)$, for adding an
affine function $\psi$ to a function $\lambda$ does not change the subdivision,
$\calS_\lambda=\calS_{\psi+\lambda}$, and elements of $\Aff(\calA)$ induce the trivial
subdivision of $\calA$ whose only facet is $\calA$.

A polyhedral subdivision $\calS$ of $\calA$  \demph{is refined} by another 
$\calS'$ (\defcolor{$\calS\prec\calS'$})
if for every face $\calF'$ of $\calS'$, there is a face $\calF$ of $\calS$ with
$\calF\supset\calF'$. 
This refinement poset is equal to the poset of the cones of the secondary fan under 
inclusion.
That is, $\calS_\sigma$ is refined by $\calS_\rho$ if and only if $\sigma$ 
is a face of $\rho$.
In particular, if $\{\ba_1,\dotsc,\ba_r\}$ is not a subset of any face of $\calS_\sigma$
then it is not a subset of any face of $\calS_\rho$ for any cone $\rho$ of the secondary
fan that contains $\sigma$.

\begin{example}\label{Ex:pentagon}
 Let $\calA=\{(0,0),(1,0),(1,1),(\frac{1}{2},\frac{3}{2}),(0,1)\}\subset\R^2$.
 Its convex hull is a pentagon.
\[
  \begin{picture}(84,64)(-18,-5)
   \put(7.5,-2){\includegraphics{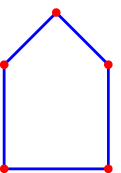}}
      \put(12,51){\footnotesize$(\frac{1}{2},\frac{3}{2})$}
   \put(-18,26){\footnotesize$(0,1)$}   \put(42,26){\footnotesize$(1,1)$}
   \put(-18,-4){\footnotesize$(0,0)$}   \put(42,-4){\footnotesize$(1,0)$}
  \end{picture}
\]
 Figure~\ref{F:allRegular} shows the poset of regular subdivisions of $\calA$.
 For each, it gives the corresponding polyhedral subdivision of $\Delta_\calA$ and
 functions $\lambda$ inducing the subdivision. 
 Working modulo $\Aff(\calA)$, we assume that a function $\lambda\in\R^\calA$ takes value
 zero at the three points where the second coordinate is positive.
 The parameter $r$ in the middle row is always positive.
\begin{figure}[htb]
  \begin{picture}(332,220)(-11,-2)
   \put(140,175){\begin{picture}(39,44)(0,-5)
     \put(7.5,-2){\includegraphics{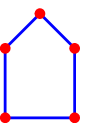}}
     \put(16.5,33){\footnotesize$0$}
     \put(1,17){\footnotesize$0$}   \put(32,17){\footnotesize$0$}
     \put(0,-5){\footnotesize$0$} \put(32,-5){\footnotesize$0$}
    \end{picture}}
  %
  \put(19,145){\includegraphics{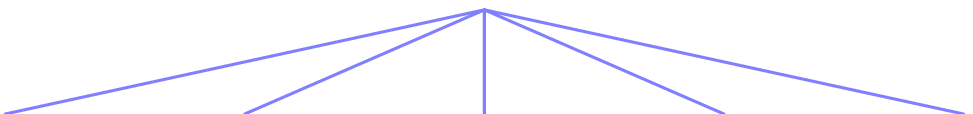}}
  %
   \put(0,97){\begin{picture}(39,44)(0,-5)
     \put(7.5,-2){\includegraphics{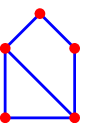}}
     \put(16.5,33){\footnotesize$0$}
     \put(1,17){\footnotesize$0$}   \put(32,17){\footnotesize$0$}
     \put(-6,-5){\footnotesize$-r$} \put(32,-5){\footnotesize$0$}
     \put(-11,10){\footnotesize$\rho_5$}
    \end{picture}}
   \put(70,97){\begin{picture}(39,44)(0,-5)
     \put(7.5,-2){\includegraphics{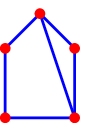}}
     \put(16.5,33){\footnotesize$0$}
     \put(1,17){\footnotesize$0$}   \put(32,17){\footnotesize$0$}
     \put(-1,-4){\footnotesize$\frac{r}{2}$} \put(32,-4){\footnotesize$r$}
     \put(-11,10){\footnotesize$\rho_1$}
    \end{picture}}
   \put(140,97){\begin{picture}(39,44)(0,-5)
     \put(7.5,-2){\includegraphics{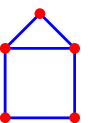}}
     \put(16.5,33){\footnotesize$0$}
     \put(1,17){\footnotesize$0$}   \put(32,17){\footnotesize$0$}
     \put(-6,-5){\footnotesize$-r$} \put(32,-5){\footnotesize$-r$}
     \put(-11,10){\footnotesize$\rho_4$}
    \end{picture}}
   \put(210,97){\begin{picture}(39,44)(0,-5)
     \put(7.5,-2){\includegraphics{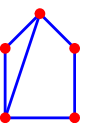}}
     \put(16.5,33){\footnotesize$0$}
     \put(1,17){\footnotesize$0$}   \put(32,17){\footnotesize$0$}
     \put(1,-4){\footnotesize$r$} \put(32,-4){\footnotesize$\frac{r}{2}$}
     \put(-11,10){\footnotesize$\rho_2$}
    \end{picture}}
   \put(280,97){\begin{picture}(39,44)(0,-5)
     \put(7.5,-2){\includegraphics{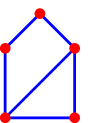}}
     \put(16.5,33){\footnotesize$0$}
     \put(1,17){\footnotesize$0$}   \put(32,17){\footnotesize$0$}
     \put(1,-5){\footnotesize$0$} \put(32,-5){\footnotesize$-r$}
     \put(-11,10){\footnotesize$\rho_3$}
    \end{picture}}
  %
  \put(19,67){\includegraphics{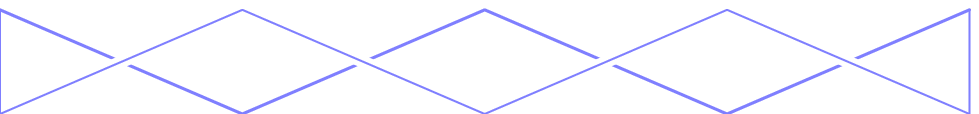}}
  %
   \put(0,0){\begin{picture}(39,64)(0,-25)
     \put(7.5,-2){\includegraphics{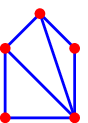}}
     \put(16.5,33){\footnotesize$0$}
     \put(1,17){\footnotesize$0$}   \put(32,17){\footnotesize$0$}
     \put(0,-5){\footnotesize$s$} \put(32,-5){\footnotesize$t$}
     \put(7,-15){\footnotesize$t>0$}
     \put(7,-25){\footnotesize$t>2s$}
     \put(-11,10){\footnotesize$\tau_5$}
    \end{picture}}
   \put(70,0){\begin{picture}(39,64)(0,-25)
     \put(7.5,-2){\includegraphics{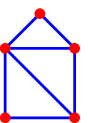}}
     \put(16.5,33){\footnotesize$0$}
     \put(1,17){\footnotesize$0$}   \put(32,17){\footnotesize$0$}
     \put(0,-5){\footnotesize$s$} \put(32,-5){\footnotesize$t$}
     \put(7,-15){\footnotesize$t<0$}
     \put(7,-25){\footnotesize$s<t$}
     \put(-11,10){\footnotesize$\tau_4$}
    \end{picture}}
   \put(140,0){\begin{picture}(39,64)(0,-25)
     \put(7.5,-2){\includegraphics{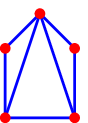}}
     \put(16.5,33){\footnotesize$0$}
     \put(1,17){\footnotesize$0$}   \put(32,17){\footnotesize$0$}
     \put(0,-5){\footnotesize$s$} \put(32,-5){\footnotesize$t$}
     \put(7,-15){\footnotesize$t>0$}
     \put(-7,-25){\footnotesize$2t>s>\frac{t}{2}$}
     \put(-11,10){\footnotesize$\tau_1$}
    \end{picture}}
   \put(210,0){\begin{picture}(39,64)(0,-25)
     \put(7.5,-2){\includegraphics{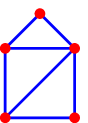}}
     \put(16.5,33){\footnotesize$0$}
     \put(1,17){\footnotesize$0$}   \put(32,17){\footnotesize$0$}
     \put(0,-5){\footnotesize$s$} \put(32,-5){\footnotesize$t$}
     \put(7,-15){\footnotesize$s<0$}
     \put(7,-25){\footnotesize$t<s$}
     \put(-11,10){\footnotesize$\tau_3$}
    \end{picture}}
   \put(280,0){\begin{picture}(39,64)(0,-25)
     \put(7.5,-2){\includegraphics{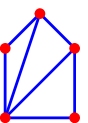}}
     \put(16.5,33){\footnotesize$0$}
     \put(1,17){\footnotesize$0$}   \put(32,17){\footnotesize$0$}
     \put(0,-5){\footnotesize$s$} \put(32,-5){\footnotesize$t$}
     \put(7,-15){\footnotesize$s>0$}
     \put(7,-25){\footnotesize$s>2t$}
     \put(-11,10){\footnotesize$\tau_2$}
    \end{picture}}
  \end{picture}
  \caption{Poset of regular subdivisions of $\calA$.}
  \label{F:allRegular}
 \end{figure}
Working modulo $\Aff(\calA)$ (using these parameters),
the secondary fan of $\calA$ is shown in
Figure~\ref{F:SecFan}. 
\begin{figure}
 \begin{picture}(95,90)(-47,-45)
  \put(-41,-41){\includegraphics{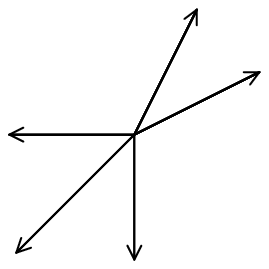}}
  \put(-47,-1){$\rho_5$}  \put(17,40){$\rho_1$}
  \put(-3.5,-43){$\rho_3$}  \put(38,17){$\rho_2$}
  \put(-43,-40){$\rho_4$}
  \put(-17,11){$\tau_5$}  \put(14,16){$\tau_1$}
  \put(-15,-24){$\tau_3$}  \put(10,-11){$\tau_2$}
  \put(-28,-11){$\tau_4$}
  \end{picture}
 \caption{Secondary fan of $\calA$.}
 \label{F:SecFan}
\end{figure}

\end{example}

\section{Real irrational toric varieties}

Let $\calA\subset\R^d$ be a finite set of vectors.
We do not assume that $\calA$ affinely spans $\R^d$.
The \demph{$\calA$-simplex} $\bsimplex^\calA\subset\R^\calA_\geq$ is the convex hull
of the standard basis vectors  $\{e_{\ba}\mid \ba\in\calA\}$ in $\R^\calA$ i.e.,
\[
   \simplex^\calA\ :=\ 
     \biggl\{z\in\R^\calA_\geq \mid \sum_{\ba\in\calA} z_{\ba}=1\biggr\}\,.
\]
It is convenient to represent points of $\simplex^\calA$ using homogeneous coordinates
$[z_{\ba}\mid \ba\in\calA]$ where each $z_{\ba}\geq 0$, not all coordinates are equal to
zero, and we have 
\[
   [z_{\ba}\mid \ba\in\calA]\ =\ 
   [\gamma z_{\ba}\mid \ba\in\calA]\qquad\mbox{for all\ }\gamma\in\R_>\,.
\]
The real torus $\R^\calA_>$ acts on $\simplex^\calA$ where for 
$w\in\R^\calA_>$ and $z\in\simplex^\calA$, we have
\[
   w.z\ :=\ [w_{\ba}\cdot z_{\ba} \mid \ba\in\calA]\,.
\]
When $\calF\subset\calA$, the simplex $\simplex^\calF$ is a face of
$\simplex^\calA$ and all faces of $\simplex^\calA$ arise in this way.
The restriction map from $\R^\calA$ to $\R^\calF$ given by $z\mapsto z|_\calF$ induces a rational map 
$\pi_\calF\colon\simplex^\calA\dasharrow\simplex^\calF$, which is undefined on 
$\simplex^{\calA\smallsetminus\calF}$.
On the remainder $\simplex^\calA\smallsetminus\simplex^{\calA\smallsetminus\calF}$, we restrict $z\in\R^\calA$ to $z|_\calF\in\R^\calF$ and
then rescale $z|_\calF$ to obtain a point in the simplex $\simplex^\calF$.

The simplex $\simplex^\calA$ is a compact metric space where we measure distance with the
$\ell_1$-metric from $\R^\calA$.
That is, if $y,z\in\simplex^\calA$, then 
\[
   \defcolor{d(y,z)}\ :=\ \sum_{\ba\in\calA} |y_\ba-z_\ba|\,.
\]

\begin{lemma}\label{L:ell_1}
 Suppose that $\calF\subset\calA$ and 
 $z\in\simplex^\calA\smallsetminus\simplex^{\calA\smallsetminus\calF}$, so that the
 projection $\pi_\calF(z)$ to $\simplex^\calF$ is defined.
 Then
\[
   d(z,\pi_\calF(z))\ =\ 2\sum_{\ba\in\calA\smallsetminus\calF} z_\ba\,.
\]
\end{lemma}

\begin{proof}
 Set $y:=\pi_\calF(z)$, which is obtained by restricting $z\in\R^\calA$ to $z|_\calF\in\R^\calF$
 and then scaling to obtain a point in the simplex $\simplex^\calF$.
 That is,
\[
    y_\ba\ =\ \left\{
     \begin{array}{ccl} 0&\mbox{if}&\ba\not\in\calF\\
        \dfrac{z_\ba}{\sum_{\fb\in\calF} z_\fb}\ &\mbox{if}& \ba\in\calF
     \end{array}\right.\ 
\]
 Note that if $\fb\in\calF$ then $y_\fb\geq z_\fb$.
 Thus,
 \begin{eqnarray*}
  d(y,z) &=& \sum_{\ba\in\calA\smallsetminus\calF} z_\ba\ +\ 
             \sum_{\ba\in\calF}\left(\frac{z_\ba}{\sum_{\fb\in\calF} z_\fb} - z_\ba\right)\\
   &=& \sum_{\ba\in\calA\smallsetminus\calF} z_\ba\ +\ 1\ -\ 
       \sum_{\ba\in\calF}z_\ba
   \ =\ 2\sum_{\ba\in\calA\smallsetminus\calF} z_\ba\,,
 \end{eqnarray*}
 as $1=\sum_{\ba\in\calA\smallsetminus\calF} z_\ba +\sum_{\ba\in\calF}z_\ba$.
\end{proof}

\subsection{Hausdorff distance}\label{S:HD}
The \demph{Hausdorff distance} beween closed subsets
$X,Y\subset\simplex^\calA$ is 
\[
   d_H(X,Y)\ :=\ 
    \max\biggl\{  \sup_{x\in X}\inf_{y\in Y} d(x,y)\,,\ 
             \sup_{y\in Y}\inf_{x\in X} d(x,y) \biggr\}\,.
\]
This endows the set of closed subsets of $\simplex^\calA$ with the structure of a complete
metric space, and the corresponding metric topology is the \demph{Hausdorff topology}.
If we have a sequence $\{X_i\mid i\in\N\}$ of subsets of $\simplex^\calA$, then 
\[
   \lim_{i\to \infty} X_i\ =\ X
\]
if and only if $X$ contains all accumulation points of the sequence $\{X_i\mid i\in\N\}$,
and each point of $X$ is a limit point of the sequence. 

\subsection{Real irrational toric varieties}

For $x\in\R_>$ and $a\in\R$, set $\defcolor{x^a}:=\exp(a\log(x))$.
For $x\in\R^d_>$ and $\ba\in\R^d$, we have the monomial
$\defcolor{x^{\ba}}:=x_1^{\ba_1}\dotsb x_d^{\ba_d}$.
The points $\calA\subset\R^d$ define a map
 \begin{equation}\label{Eq:parametrization}
   \defcolor{\varphi_\calA}\ \colon\ \R^d_>\ \longrightarrow\ \simplex^\calA
   \qquad\mbox{where}\qquad 
    \varphi_\calA(x)\ =\ [x^{\ba}\mid \ba\in\calA]\,.
 \end{equation}
The \demph{real irrational toric variety $X_\calA$} is the closure of the image
\defcolor{$X_\calA^\circ$} of $\varphi_\calA$ in $\simplex^\calA$.

The convex hull \defcolor{$\Delta_\calA$} of $\calA$ 
is the image of $\simplex^\calA$ under the map 
$\defcolor{\taut_\calA}\colon\R^\calA \to \R^d$ defined by
\[
   \taut_\calA\ \colon\ 
    (z_{\ba}\mid \ba\in\calA)\ \longmapsto\ \sum_{\ba\in\calA} z_\ba \ba\ .
\]
The following theorem of Birch from algebraic statistics~\cite[p.~168]{Agr90} 
identifies $X_\calA$ with $\Delta_\calA$.

\begin{theorem}
  The restriction of $\taut_\calA$ to $X_\calA$ is a homeomorphism
 $\taut_\calA\colon X_\calA\to\Delta_\calA$.
\end{theorem}

In particular this shows that the toric variety $X_\calA$ has dimension equal
to the dimension of the convex hull of $\calA$.
We call this restriction of $\taut_\calA$ to $X_\calA$ the 
\demph{algebraic moment map}.

Homogeneous equations for $X_\calA$ were described in~\cite[Prop.~B.3]{CGS} as follows.
For every affine relation among the points of $\calA$ with nonnegative coefficients
 \begin{equation}\label{Eq:N-relation}
   \sum_{\ba\in\calA} \alpha_\ba \ba\ =\ \sum_{\ba\in\calA} \beta_\ba \ba\,
    \qquad\mbox{where}\qquad
   \sum_{\ba\in\calA}\alpha_\ba \ =\ \sum_{\ba\in\calA} \beta_\ba\,,
 \end{equation}
with $\alpha_\ba,\beta_\ba\in\R_>$,
we have the valid equation for points $z\in X_{\calA}$,
 \begin{equation}\label{Eq:poly-relation}
   \prod_{\ba\in\calA}z_\ba^{\alpha_\ba}\ =\
   \prod_{\ba\in\calA}z_\ba^{\beta_\ba}\,.
 \end{equation}
Conversely, if $z\in\simplex^\calA$ satisfies equation~\eqref{Eq:poly-relation} for
every affine relation~\eqref{Eq:N-relation}, then $z\in X_{\calA}$.

Given a point $w=(w_\ba\mid \ba\in\calA)\in \R_>^\calA$ of the real torus, we have the
translated toric variety $\defcolor{X_{\calA,w}}:=w.X_{\calA}$, which is the closure of
$\defcolor{X^\circ_{\calA,w}}:=w.X^\circ_\calA$.
Birch's Theorem still holds for $X_{\calA,w}$; it is mapped homeomorphically to
$\Delta_\calA$ by the algebraic moment map $\taut_\calA$. 
We have the following description of the equations for $X_{\calA,w}$. 

\begin{proposition}\label{P:equations}
  A point $z\in \simplex^\calA$ lies in $X_{\calA,w}$ if and only if
 \[
   \prod_{\ba\in\calA}z_\ba^{\alpha_\ba} \cdot \prod_{\ba\in\calA}w_\ba^{\beta_\ba}
   \ =\
   \prod_{\ba\in\calA}z_\ba^{\beta_\ba} \cdot \prod_{\ba\in\calA}w_\ba^{\alpha_\ba}\,,
 \]
 for every affine relation~\eqref{Eq:N-relation} among the points of $\calA$.
 On $X^\circ_{\calA,w}$, we additionally have such equations coming from affine
 relations~\eqref{Eq:N-relation} where the numbers $\alpha_\ba,\beta_\ba$ are allowed to be
 negative. 
\end{proposition}

The real torus $\R^\calA_>$ acts on $X_\calA$ with the stabilizer of $X_\calA$
the image of $\R^{d+1}_>$ in $\R^\calA_>$ under the map
\[
   (t_0,t_1,\dotsc,t_d)\ \mapsto (t_0 t^\ba\mid \ba\in\calA)\,.
\]
Under the coordinatewise logarithm map $\defcolor{\Log}\colon\R^\calA_>\to\R^\calA$, this
stabilizer subgroup is mapped to the subspace $\Aff(\calA)$ of affine functions on
$\calA$. 

\begin{lemma}\label{L:stabilizer}
For $w,w'\in\R^\calA_>$,  $X_{\calA,w}=X_{\calA,w'}$ if and only if 
 $\Log(w)-\Log(w')\in\Aff(\calA)$.
\end{lemma}

The toric variety $X_\calF\subset\simplex^\calF$ is the image of the toric variety
$X_\calA$ under the map $\pi_\calF$.
This can be seen either from the definition~\eqref{Eq:parametrization} or from the
equations~\eqref{Eq:poly-relation} for $X_\calA$.
Likewise, if $w\in\R^\calA_>$ and $w|_\calF$ is its restriction to $\calF$,
then $\pi_\calF(X_{\calA,w})= X_{\calF,w|_\calF}$.
We write $X_{\calF,w}$ for $X_{\calF,w|_\calF}$ to simplify notation.
When $\Delta_\calF$ has dimension $d$, the map 
$\pi_\calF\colon X^\circ_{\calA,w}\to X^\circ_{\calF,w}$ is a bijection.

A consequence of the properties of $\taut_\calA$ is a
description of the boundary of $X_{\calA,w}$. 
Let $F$ be a face of the polytope $\Delta_\calA$ and 
$\calF = \calA\cap F$ be the points of $\calA$ lying on $F$,  also called a
\demph{face} of $\calA$.
Then the toric variety $X_{\calF,w}$ equals $X_{\calA,w}\cap\simplex^\calF$.
The collection of toric varieties $X_{\calF,w}$ where $F$ ranges over the 
faces of
$\Delta_\calA$ forms the boundary of $X_{\calA,w}$.
We have the decomposition of $X_{\calA,w}$ into disjoint subsets,
 \begin{equation}\label{Eq:disjont_deomposition}
   X_{\calA,w}\ =\ \bigsqcup_{\calF} X^\circ_{\calF,w}\,,
 \end{equation}
where $\calF$ ranges over all faces of $\calA$.
Each set $X^\circ_{\calF,w}$ is an orbit of $\R^\calA_>$ acting on $X_{\calA,w}$.

\section{Toric degenerations of real irrational toric varieties}\label{translates and toric degenerations}

We describe all limits of the toric variety $X_\calA$ under cosets of one-parameter subgroups of
$\R^\calA_>$, called \demph{toric degenerations}. 
Each limit is a complex of toric varieties supported on a union of faces of
$\simplex^\calA$ that is the geometric realization of a regular subdivision of
$\calA$.
Its proof follows that of Theorem~A.1 in~\cite{GSZ} which was
for the case when $\calA\subset\Z^d$.

\subsection{Complexes of toric varieties}
Let $\calS$ be a polyhedral subdivision of $\calA$.
The \demph{geometric realization $|\calS|$} of $\calS$ is the  union
\[
  \defcolor{|\calS|}\ :=\  \bigcup_{\calF\mbox{\scriptsize\ a face of }\calS} \simplex^\calF
\]
of faces of $\simplex^\calA$ corresponding to faces of the subdivision $\calS$.
The following is standard, it holds for more general simplicial complexes on $\calA$
(see for instance \cite[Chapter 1]{MSt}).

\begin{proposition}\label{P:geometric_realization}
 The geometric realization $|\calS|$ of a polyhedral subdivision $\calS$ of $\calA$
 is defined as an analytic subset of $\simplex^\calA$ 
 by the vanishing of the monomials 
\[
  \{ z_\ba z_\bb\,\mid\, \mbox{$\{\ba,\bb\}$ is not a subset of any face of  $\calS$}\}
    \ \bigcup\ 
  \{ z_\bc\,\mid\, \mbox{$\bc$ lies in no face of $\calS$}\}\,.
\]
\end{proposition}

For a polyhedral subdivision $\calS$ of $\calA$, the corresponding union of toric varieties
 \begin{equation}\label{Eq:complex}
   \defcolor{X(\calS)}\ :=\ \bigcup_{\calF\mbox{\scriptsize\ a face of }\calS} X_\calF
 \end{equation}
is the \demph{complex of toric varieties corresponding to $\calS$}.
This is the union of toric varieties $X_\calF$ for $\calF$ a facet of $\calS$
glued together along toric subvarieties corresponding to common faces.
That is, if $\calG=\calF\cap\calF'$, then $\Delta_\calG=\Delta_\calF\cap\Delta_{\calF'}$ is
a common face and $X_\calG=X_\calF\cap X_{\calF'}$.

A point $w\in\R^\calA_>$ of the positive torus 
acts on the complex of toric varieties~\eqref{Eq:complex} by translation,
giving the translated complex,
 \begin{equation}\label{Eq:complex_translated}
   \defcolor{X(\calS,w)}\ :=\ w.X(\calS)\ =\ 
   \bigcup_{\calF\mbox{\scriptsize\ a face of }\calS} X_{\calF,w}\,.
 \end{equation}
A consequence of~\eqref{Eq:complex_translated}
and the decomposition~\eqref{Eq:disjont_deomposition} of $X_\calA$ into disjoint orbits is
the decomposition of $X(\calS,w)$ into disjoint orbits,
\[
   X(\calS,w)\ =\ \bigsqcup_{\calF\mbox{\scriptsize\ a face of }\calS} X^\circ_{\calF,w}\,.
\]
The union of $X^\circ_{\calF,w}$ where $\calF$ ranges over the facets of $\calS$ forms
a dense open subset of the complex  $X(\calS,w)$ of toric varieties.

\subsection{Toric degenerations}

An element $\lambda\in\R^\calA$ defines a one-parameter subgroup \defcolor{$\lambda(t)$}
in $\R^\calA$: for $t\in\R$ set $\lambda(t)_{\ba}:= \exp(t\lambda(\ba))$.
Given $w\in\R^{\calA}$, we have the coset
$\defcolor{w_\lambda(t)}:=\lambda(t)\cdot w$ and the corresponding family of translated
toric varieties $X_{\calA,w_\lambda(t)}= \lambda(t). X_{\calA,w}$.
We omit the proof of the following theorem, which is the same, {\it mutatis mutandis}, as that of the corresponding  
statement (Theorem~3) in~\cite{GSZ}.

\begin{theorem}\label{Th:toric_degenerations}
 Let $\lambda\in\R^\calA$.
 For any $w\in\R^\calA_>$, the family $\lambda(t).X_{\calA,w}$ of translated toric
 varieties has a limit as $t\to\infty$ in the Hausdorff topology on closed subsets of
 $\simplex^\calA$, and 
\[
   \lim_{t\to\infty} \lambda(t).X_{\calA,w}\ =\ X(\calS,w)\,,
\]
 where $\calS$ is the regular subdivision of $\calA$ induced by $\lambda$.
\end{theorem}

This theorem shows that when $\calS$ is the regular subdivision induced by the function
$\lambda$, the complex of toric varieties $X(\calS,w)$ is the limit of a sequence 
of translates of the real irrational toric variety $X_{\calA,w}$ by elements of the
one-parameter subgroup $\lambda(t)$.

In \cite[Th.~5.2]{GSZ} a weak converse of Theorem~\ref{Th:toric_degenerations}
was proved when $\calA\subset\Z^d$:
if a sequence of translates of $X_\calA$  has a limit in
the Hausdorff topology, then this limit is the complex of toric varieties
$X(\calS,w)$ for some regular subdivision $\mathcal{S}$ of $\calA$ and $w\in\R^\calA_>$. 

We prove a stronger result, that {\it every} sequence of
translates has a subsequence that converges in the Hausdorff topology
to some complex $X(\calS,w)$ of toric varieties.
This implies that the set of translates of $X_\calA$ is compactified by the
set of toric degenerations of $X_\calA$.

\begin{theorem}\label{Th:creatingLimits}
 For every finite set $\calA\subset\R^d$ and every sequence $\{w_i\mid i\in\N\}$ in the
 positive torus $\R^\calA_>$, there is a subsequence $\{u_j\mid j\in\N\}\subset\{w_i\mid
 i\in\N\}$, a regular subdivision $\calS$ of $\calA$,  and an element $w\in\R^\calA_>$
 such that
\[
    \lim_{j\to\infty} u_j.X_\calA\ =\ X(\calS,w)
\]
 in the Hausdorff topology on subsets of the simplex $\simplex^\calA$.
\end{theorem}

The subsequence $\{u_j\mid j\in\N\}$ and regular subdivision $\calS$ of Theorem~\ref{Th:creatingLimits} are
obtained as follows.
The sequence of logarithms  $\{\Log(u_j)\mid j\in\N\}$ is a subsequence of $\{\Log(w_i)\mid i\in\N\}$
that lies in a cone $\tau$ of the secondary fan of $\calA$ having a face $\sigma$ that is the minumum face of
boundedness of any subsequence of   $\{\Log(u_j)\mid j\in\N\}$, as in Lemma~\ref{L:finalSubdivision}.
Then the subdivision $\calS$ is $\calS_\sigma$, the regular subdivision of $\calA$ 
corresponding to the cone $\sigma$.
More details are given in Subsection~\ref{S:limit_set}, where we also define the weight $w$.

The proof of Theorem~\ref{Th:creatingLimits} occupies Section~\ref{algorithm}.

%
\section{Hausdorff limits of torus translates}\label{algorithm}

Let $\{w_i\mid i\in\N\}\subset\R^\calA_>$ be a sequence of elements of the positive torus.
Consider the corresponding sequence of logarithms,
$\defcolor{v_i}:=\Log(w_i)$ for $i\in\N$.
We show the existence of a subsequence of $\{w_i\}$ (equivalently of 
$\{v_i\}$) so that the corresponding sequence of torus translates $X_{\calA,w_i}$ of
$X_{\calA}$ converges in the Hausdorff topology to a complex of toric varieties
$X(\calS,w)$  for some regular subdivision $\calS$ (which we construct) of $\calA$ and a weight
$w\in\R^\calA_>$ (which we also identify). 
For this, we will freely replace the sequence $\{v_i\}$ by subsequences
throughout.
Let us begin with an example.

\begin{example}\label{Ex:convergence-trivial}
 Suppose that the sequence $\{v_i\mid i\in\N\}$ has an accumulation point modulo
 $\Aff(\calA)$.
 Replacing $\{v_i\}$ by a subsequence, we may assume that it is convergent
 modulo $\Aff(\calA)$.
 Convergent sequences are bounded so there is a bounded set $B\subset\R^\calA$
 with $\{v_i\mid i\in\N\}\subset \Aff(\calA)+B$ and therefore a sequence
 $\{u_i\mid i\in\N\}\subset\Aff(\calA)$ and a convergent sequence
 $\{\overline{v}_i\mid i\in\N\}\subset B$ such
 that $v_i=u_i+\overline{v}_i$, for each $i\in\N$.
 Let \defcolor{$v$} be the limit of the sequence $\{\overline{v}_i\}$.

 Set $\defcolor{\overline{w}_i}:=\Exp(\overline{v}_i)$ and 
 $\defcolor{w}:=\Exp(\overline{v})$.
 As $v_i-\overline{v}_i\in\Aff(\calA)$, we have 
 $X_{\calA,w_i}=X_{\calA,\overline{w}_i}$.
 Then, as  $\lim_{i\to\infty}\overline{w}_i=w$, the equations of Proposition~\ref{P:equations}
 for $X_{\calA,w_i}$ show that  
\[
   \lim_{i\to \infty}  w_i.X_{\calA}\ =\ 
   \lim_{i\to \infty}  \overline{w}_i.X_{\calA}\ =\ w.X_{\calA}\,,
\]
 In this case, the limit of torus translate is just another torus translate. 
\end{example}

%
\subsection{The limiting set $X(\calS,w)$}
\label{S:limit_set}
We replace $\{v_i\mid i\in\N\}$ by a subsequence from which we determine a regular
subdivision $\calS$ of $\calA$ and a weight $w\in\R^\calA_>$.
These define a complex $X(\calS,w)$  of toric varieties which we will show is the limit of the
sequence of torus translates $X_{\calA,w_i}$ corresponding to that subsequence.

The secondary fan $\Sigma_\calA$ consists of finitely many cones $\sigma$.
Let $\tau\in\Sigma_\calA$ be a cone which is minimal under inclusion such that
$\tau\cap\{v_i\mid i\in\N\}$ is infinite.
Replacing $\{v_i\}$ by $\tau\cap\{v_i\}$, we have that
$\{v_i\}\subset\tau$ and if $\sigma\subsetneq\tau$, then 
$\{v_i\}\cap\sigma$ is finite.

By Lemma~\ref{L:finalSubdivision}, replacing $\{v_i\}$ by a
subsequence if necessary, there is a face $\sigma$ of $\tau$ that is the minimum face of boundedness of any
subsequence of $\{v_i\}$. 
Let $\calS:=\calS_\sigma$ be the regular subdivision of $\calA$ corresponding to $\sigma$.

As $\{v_i\}$ is $\sigma$-bounded, its image in
$\R^{\calA}/\langle\sigma\rangle$ is bounded, and so there exists a closed bounded
set $B\subset\tau$ with $\{v_i\}\subset\sigma+ B$.
Thus there are sequences $\{u_i\}\subset\sigma$ and 
$\{\overline{v}_i\}\subset B$ with 
\[
    v_i\ =\ u_i+\overline{v}_i
   \qquad\mbox{for}\quad i\in\N\,.
\]
The sequence $\{\overline{v}_i\}$ has an accumulation
point $v\in B$.
Replacing $\{\overline{v}_i\}$ by a subsequence we have that
\[
   \lim_{i\to\infty} \overline{v}_i\ =\ v\,,
\]
and replace $\{u_i\}$ and $\{v_i\}$ by the corresponding subsequences. 
We define the vector \demph{$w$} by $w_\ba:=\exp(v_\ba)$ and write 
$w=\defcolor{\Exp}(v)$.

We show that $X(\calS,w)$ is the limit of the sequence of translations of the complex 
$X(\calS)$ of toric varieties by the sequence 
$\{w_i\mid i\in\N\}=\{\Exp(v_i)\mid i\in\N\}$.

\begin{lemma}\label{L:complex_limit}
  In the Hausdorff topology we have
\[
    \lim_{i\to\infty}  X(\calS,w_i)\ =\ X(\calS,w)\,.
\]
\end{lemma}

\begin{proof}
 Set $\overline{w}_i:=\Exp(\overline{v}_i)$.
 Let $\calF$ be a face of $\calS=\calS_\sigma$.
 Then $\overline{v}_i-v_i=u_i\in\sigma$, and so the restriction of each $u_i$ to $\calF$ is an
 affine function.
 Thus $X_{\calF,w_i}=X_{\calF,\overline{w}_i}$, by Lemma~\ref{L:stabilizer}.
 Since $\overline{v}_i\to v$ as $i\to\infty$, we have 
 $\overline{w}_i\to w$ as $i\to\infty$, and thus
 \begin{equation}\label{Eq:limit_F}
   \lim_{i\to\infty} X_{\calF,w_i}\ =\ X_{\calF,w}\,.
 \end{equation}
 This proves the lemma, by the definition~\eqref{Eq:complex_translated}, 
 $X(\calS,w_i)$ and $X(\calS,w)$ are the union of 
 $X_{\calF,w_i}$ and $X_{\calF,w}$ for $\calF$ a face of $\calS$, respectively.
\end{proof}

\begin{example}\label{Ex:final_weights}
For the point configuration 
$\calA:=\{(0,0),(1,0),(1,1),(\frac{1}{2},\frac{3}{2}),(0,1)\}$ in $\R^2$ of
Example~\ref{Ex:pentagon}, consider the sequence 
$\{w_i \mid i\in\N\}\subset \R^{\calA}_>$ where $w_i:=\Exp(v_i)$ and 
\[
   \defcolor{v_i}\ :=\ 
   \left(-i{-}\tfrac{1}{i}\,,\,i{-}1 \,,\,i \,,\,-\tfrac{i}{2}\,,\,-i \right)\,.
\]
An affine function on $\R^2$ is given by $(x,y)\mapsto a+bx+cy$, and so this sequence is
equivalent modulo $\Aff(\calA)$ to any sequence of the form
\[
  \left( a_i{-}i{-}\tfrac{1}{i}\,,\, a_i{+}b_i{+}i{-}1 \,,\, a_i{+}b_i{+}c_i{+}i \,,\,
         a_i{+}\tfrac{b_i}{2}{+}\tfrac{3c_i}{2}{-}\tfrac{i}{2} \,,\,a_i{+}c_i{-}i\right)\,.
\]
Setting $a_i=0$, $b_i=-2i$, and $c_i=i$, we obtain the equivalent sequence
\[
   \widetilde{v}_i\ :=\ 
    \left(  -i{-}\tfrac{1}{i} \,,\, -i{-}1 \,,\,0 \,,\, 0 \,,\, 0 \right)\,,
\]
which lies in the plane used in Example~\ref{Ex:pentagon} for representatives of
$\R^\calA$ modulo $\Aff(\calA)$.

In Figure~\ref{F:valuesOfSeq}, we show the coordinates of $v_i$ and 
$\widetilde{v}_i$, together with the induced triangulation, which is the same for all
$i>1$. 
\begin{figure}[htb]
  \begin{picture}(84,60)(-19,-5)
   \put(7.5,-2){\includegraphics{figures/pentagonBig.eps}}
     \put(-10,45){\footnotesize$v_i$}
      \put(15,52){\footnotesize$-\frac{i}{2}$}
   \put( -6,26){\footnotesize$-i$}   \put(42,26){\footnotesize$i$}
   \put(-19,-4){\footnotesize$-i{-}\frac{1}{i}$}  
         \put(42,-4){\footnotesize$i{-}1$}
  \end{picture}
    \qquad
  \begin{picture}(84,60)(-19,-5)
   \put(7.5,-2){\includegraphics{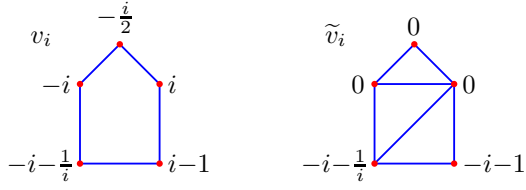}}
     \put(-10,45){\footnotesize$\widetilde{v}_i$}
      \put(21,48){\footnotesize$0$}
   \put(  0,26){\footnotesize$0$}   \put(42,26){\footnotesize$0$}
   \put(-19,-4){\footnotesize$-i{-}\frac{1}{i}$}  
         \put(42,-4){\footnotesize$-i{-}1$}
  \end{picture}
 \caption{Triangulation induced by $\{v_i\mid i\in\N\}$.}
 \label{F:valuesOfSeq}
\end{figure}
Thus we see that each $\widetilde{v}_i$ and also $v_i$ lies in the full-dimensional cone
$\tau_3$ of the secondary fan $\Sigma_\calA$ (see Figures~\ref{F:allRegular}
and~\ref{F:SecFan}).
In the coordinates $\R^2$ for $\R^\calA/\Aff(\calA)$, we have 
$\widetilde{v}_i=(-i-\tfrac{1}{i},-i{-}1)$  and the rays $\rho_3$ and $\rho_4$ of
$\tau_3$ are generated by $e_3:=(0,-1)$ and $e_4:=(-1,-1)$,
respectively.
Writing $\widetilde{v}_i$ in this basis for $\R^2$ gives
\[
   \widetilde{v}_i\ :=\ (1{-}\tfrac{1}{i})e_3 \ +\ (i{+}\tfrac{1}{i}) e_4\,.
\] 
Thus the images of $\{v_i\}$ in the quotients 
$\R^\calA/\langle\rho_3\rangle=\R^2/\R e_3\simeq \R e_4$ 
and $\R^\calA/\langle\rho_4\rangle= \R^2/\R e_4\simeq \R e_3$ are, respectively,
\[
   \{(i{+}\tfrac{1}{i})e_4\}
  \qquad\mbox{ and }\qquad
   \{(1{-}\tfrac{1}{i})e_3\}\,.
\]
The first is divergent while the second is bounded.

The minimum face of boundedness of any subsequence of $\{v_i\}$ is $\rho_4$.
If we set
\[
  u_i\ :=\ (-i-\frac{1}{i}\,,\,-i-\frac{1}{i}\,,\,0\,,\,0\,,\,0)
   \qquad\mbox{and}\qquad
  \overline{v}_i\ :=\ (0\,,\,-1+\frac{1}{i}\,,\,0\,,\,0\,,\,0)\,,
\]
then $ \widetilde{v}_{i}=u_i+\overline{v}_i$, where $u_i\in\rho_4$ and 
$\{\overline{v}_i\}$ is bounded in $\tau_3$.
Then $v=(0,-1,0,0,0)$ and thus 
$w:=(1,\frac{1}{e},1,1,1)$.
We display $\widetilde{v}_i$ for $i=1,\dotsc,6$, 
$v+\langle\rho_4\rangle$, and $v$ below.
\[
  \begin{picture}(140,81)(-1,0)
   \put(0,0){\includegraphics{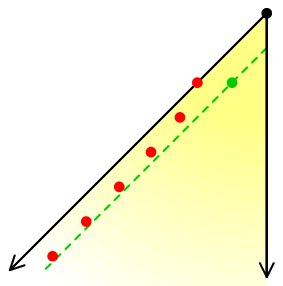}}
   \put(-1,15){$\rho_4$}  \put(48,13){$\tau_3$}  \put(83,14){$\rho_3$}
   \put(68,51){$v$}
   \put(98,60.5){$v+\langle\rho_4\rangle$}
   \put(95,63.3){\vector(-1,0){20}}

   \put(12,56){$\widetilde{v}_i$}
   \put(23,60){\vector(1,0){34}}
   \put(23,58){\vector(4,-1){27}}
   \put(23,55.8){\vector(3,-2){19}}
   \put(22,53.5){\vector(2,-3){12}}
   \put(20,52){\vector(1,-4){6}}
   \put(18.5,52){\vector(0,-1){40}}
  \end{picture}
\]

We determine the limit of the sequence 
$\{X_{\calA,w_i}\mid i\in\N\}$.
By Proposition~\ref{P:equations}, for $w\in\R^\calA$, $X_{\calA,w}$  
is defined by the vanishing of the five homogenous binomials,
\begin{eqnarray}\nonumber
   &w_{10}w_{01}z_{00}z_{11}-w_{00}w_{11}z_{10}z_{01}\,,&\\ \label{Eq:HB}
  &w_{00}w_{11}^3z_{10}^2z_{\frac{1}{2}\frac{3}{2}}^2
          -w_{10}^2w_{\frac{1}{2}\frac{3}{2}}^2z_{00}z_{11}^3\,,\ 
   w_{10}w_{01}^3z_{00}^2z_{\frac{1}{2}\frac{3}{2}}^2
          -w_{00}^2w_{\frac{1}{2}\frac{3}{2}}^2z_{10}z_{01}^3\,,&\\
  &w_{01}^2w_{11}z_{00}z_{\frac{1}{2}\frac{3}{2}}^2
          -w_{00}w_{\frac{1}{2}\frac{3}{2}}^2z_{01}^2z_{11}\,,\ 
   w_{01}w_{11}^2z_{10}z_{\frac{1}{2}\frac{3}{2}}^2
          -w_{10}w_{\frac{1}{2}\frac{3}{2}}^2z_{01}z_{11}^2&\nonumber
\end{eqnarray}
Set $w_i=\Exp(v_i)=(e^{-i-\frac{1}{i}},e^{i-1},e^i,e^{-\frac{1}{2}},e^{-i})$ and 
consider the sequence $\{X_{\calA,w_i}\}$.
We invite the reader to check that if $w\in\Exp(\Aff(\calA))$, then the coefficients of the 
monomial terms in each binomial of~\eqref{Eq:HB} are equal and therefore 
$X_{\calA}=X_{\calA,w}$.
It follows that for each $i$, $X_{\calA,w_i}\simeq X_{\calA,\widetilde{w}_i}$, where
\[
  \widetilde{w}_i\ :=\ \Exp(\widetilde{v}_i)\ =\ 
  (e^{-i-\frac{1}{i}}, e^{-i-1},  1, 1,  1)\,.
\]
 Then $X_{\calA,\widetilde{w}_i}$ is defined by the binomials
\begin{eqnarray*}\nonumber
   &e^{-i-1}z_{00}z_{11}-e^{-i-\frac{1}{i}}z_{10}z_{01}\,,\ 
  e^{-i-\frac{1}{i}}z_{10}^2z_{\frac{1}{2}\frac{3}{2}}^2
          -e^{-2i-2}z_{00}z_{11}^3\,,\ 
   e^{-i-1}z_{00}^2z_{\frac{1}{2}\frac{3}{2}}^2
          -e^{-2i-\frac{2}{i}}z_{10}z_{01}^3\,,&\\
  &z_{00}z_{\frac{1}{2}\frac{3}{2}}^2
          -e^{-i-\frac{1}{i}}z_{01}^2z_{11}\,,\ 
   z_{10}z_{\frac{1}{2}\frac{3}{2}}^2
          -e^{-i-1}z_{01}z_{11}^2&\nonumber
\end{eqnarray*}
We may rewrite the first three as 
\[
    e^{-1+\frac{1}{i}}z_{00}z_{11}-z_{10}z_{01}\,,\ 
    z_{10}^2z_{\frac{1}{2}\frac{3}{2}}^2
          -e^{-i-2+\frac{1}{i}}z_{00}z_{11}^3\,,\ 
   z_{00}^2z_{\frac{1}{2}\frac{3}{2}}^2
          -e^{-i+1-\frac{2}{i}}z_{10}z_{01}^3\,.
\]
Then, if we let $i\to\infty$, these five binomials become one binomial and
 two monomials,
\[
   e^{-1}z_{00}z_{11}-z_{10}z_{01}\,,\ 
   z_{10}^2z_{\frac{1}{2}\frac{3}{2}}^2\,,\ 
   z_{00}^2z_{\frac{1}{2}\frac{3}{2}}^2\,.
\]
The monomials define the subdivision $\calS$ of $\calA$ corresponding to the ray $\rho_4$, and
the binomial defines the toric variety $X_{\calF,w}$, where $\calF$ is the facet of the
subdivision consisting of the points $\calA\smallsetminus\{(\frac{1}{2},\frac{3}{2})\}$.
In particular, this computation implies that 
\[
   \lim_{i\to\infty} X_{\calA,w_i}\ =\ X(\calS,w)\,,
\]
which shows the conclusion of Theorem~\ref{Th:creatingLimits} for this example.
\end{example}

\begin{example}
 Example~\ref{Ex:final_weights} involved torus translations corresponding to the sequence $\{v_i\}$ of
 Example~\ref{Ex:mmfb}.
 We now consider the limit of torus translations corresponding to the sequence $\{u_i\}$  of
 Example~\ref{Ex:mmfb}.
 For $i\in \N$, set
\[
   \defcolor{w_i}\ :=\ ( e^{\sqrt{i}-i}, e^{-i}, 1,1,1)\,,
\]
 and $u_i:=\Log(w_i)=(\sqrt{i}-i,-i,0,0,0)$, which is essentially the sequence of $\{v_i\}$ of
 Example~\ref{Ex:mmfb}.
 Then $X_{\calA,w_i}$ is defined by the binomials
\begin{eqnarray*}\nonumber
   &e^{-i}z_{00}z_{11}-e^{\sqrt{i}-i}z_{10}z_{01}\,,\ 
  e^{\sqrt{i}-i}z_{10}^2z_{\frac{1}{2}\frac{3}{2}}^2
          -e^{-2i}z_{00}z_{11}^3\,,\ 
   e^{-i}z_{00}^2z_{\frac{1}{2}\frac{3}{2}}^2
          -e^{2\sqrt{i}-2i}z_{10}z_{01}^3\,,&\\
  &z_{00}z_{\frac{1}{2}\frac{3}{2}}^2
          -e^{\sqrt{i}-i}z_{01}^2z_{11}\,,\ 
   z_{10}z_{\frac{1}{2}\frac{3}{2}}^2
          -e^{-i}z_{01}z_{11}^2&\nonumber
\end{eqnarray*}
Rewriting the first three gives 
\[
    e^{-\sqrt{i}}z_{00}z_{11}-z_{10}z_{01}\,,\ 
    z_{10}^2z_{\frac{1}{2}\frac{3}{2}}^2
          -e^{-i-\sqrt{i}}z_{00}z_{11}^3\,,\ 
   z_{00}^2z_{\frac{1}{2}\frac{3}{2}}^2
          -e^{2\sqrt{i}-i}z_{10}z_{01}^3\,.
\]
 Then, as $i\to\infty$, these five binomials become the three monomials, 
\[
   z_{10}z_{01}\,,\ 
   z_{10}z_{\frac{1}{2}\frac{3}{2}}^2\,,\ 
   z_{00}z_{\frac{1}{2}\frac{3}{2}}^2\,.
\]
These monomials define the 
geometric realization $|S|$ of
the triangulation $\calS$ of $\calA$ corresponding to the cone $\tau_3$.
In particular, 
\[
   \lim_{i\to\infty} X_{\calA,w_i}\ =\ X(\calS,w)\,,
\]
for any weight $w\in\R^\calA_>$,
which shows the conclusion of Theorem~\ref{Th:creatingLimits} for this example.
\end{example}

%
\subsection{Hausdorff limits of translates}\label{limits}

We prove Theorem~\ref{Th:creatingLimits}, that
\[
   \lim_{i\to\infty} X_{\calA,w_i}\ =\ X(\calS,w)\,.
\]
We prove this limit in three steps:
\begin{enumerate}
 \item Any accumulation point of the sequence $\{X_{\calA,w_i}\mid i\in\N\}$ lies in the
       geometric realization $|\calS|$ of the regular subdivision $\calS$.
       (Lemma~\ref{L:first_third}.)
 \item For each face $\calF$ of $\calS$, any accumulation points of 
       $\{X_{\calA,w_i}\mid i\in\N\}$ in $\simplex^\calF$ lie in $X_{\calF,w}$.
       (Lemma~\ref{L:second_third}.)
 \item Every point of $X(\calS,w)$ is a limit point of $\{X_{\calA,w_i}\mid i\in\N\}$.
       (Lemma~\ref{L:third_third}.)
\end{enumerate}

By (i) and (ii) the accumulation points of $\{X_{\calA,w_i}\}$ are a subset of $X(\calS,w)$.
This, together with (iii), establishes the limit as observed in~\S~\ref{S:HD}.

\begin{lemma}\label{L:first_third}
 Let $y\in\simplex^\calA$ be an accumulation point of 
 $\{X_{\calA,w_i}\mid i\in\N\}$.
 Then $y\in|\calS|$.
\end{lemma}

\begin{proof}
We show that no point of $\simplex^\calA\smallsetminus|\calS|$ is
 an accumulation point of  $\{X_{\calA,w_i}\mid i\in\N\}$.
 Let $y\not\in|\calS|$.
 We will produce an $\epsilon>0$ and an $N$ such that if $i>N$ then 
 $d(y,X_{\calA,w_i})>\epsilon$.

 By Proposition~\ref{P:geometric_realization}, either there are $\ba,\bb\in\calA$ 
 with $\{\ba,\bb\}$ not a subset of any face of $\calS$ and $y_{\ba}y_{\bb}\neq 0$, or else
 there is a $\bc\in\calA$ that lies in no face of $\calS$ and $y_{\bc}\neq 0$.
 In the first case, set $\epsilon:=\frac{1}{2}\min\{y_{\ba},y_{\bb}\}$ and in the second
 case, set $\epsilon:=\frac{1}{2}y_{\bc}$.
 Suppose that we are in the first case.
 Since $\{\ba,\bb\}$ is not a subset of any face of $\calS_\sigma$, $\sigma$ is a face of
 $\tau$, and $S=S_\sigma$ is refined by $S_\tau$, $\{\ba,\bb\}$ is not a subset of any face of $\calS_\tau$.
 By Lemma~\ref{L:convex_combinations} there is a relation~\eqref{Eq:convex_ii} expressing a
 point $p$ in the interior of the segment $\overline{\ba,\bb}$ as a convex combination of 
 the points in a facet $\calG$ of $\calS_\tau$.
 By Proposition~\ref{P:equations} this gives the valid equation on points $z\in X_{\calA,w_i}$,
\[
    z_{\ba}^{\beta_{\ba}}z_{\bb}^{\beta_{\bb}}\ =\ 
    \frac{ w_i(\ba)^{\beta_{\ba}} w_i(\bb)^{\beta_{\bb}}}%
         {\prod_{\bg\in\calG} w_i(\bg)^{\alpha_{\bg}}}
    \cdot \prod_{\bg\in\calG} z_\bg^{\alpha_{\bg}}\,,
\]
 where we write $w_i(\ba)$ for $(w_i)_{\ba}$.
 As $0\leq z_{\bg}\leq 1$, $\alpha_\bg\geq 0$, and $w_i=\Exp(v_i)$, we have 
\[
    z_{\ba}^{\beta_{\ba}}z_{\bb}^{\beta_{\bb}}\ \leq\ 
    \exp\Bigl(\beta_{\ba} v_i(\ba) + \beta_{\bb} v_i(\bb) \ -\ 
       \sum_{\bg\in\calG} \alpha_{\bg}v_i(\bg)\Bigr)\,.
\]
 It suffices to show that the exponential has limit 0, which is equivalent to 
 \begin{equation}\label{Eq:quantity}
  \lim_{i\to\infty}
    \Bigl(\beta_{\ba} v_i(\ba) + \beta_{\bb} v_i(\bb) \ -\ 
       \sum_{\bg\in\calG} \alpha_{\bg}v_i(\bg)\Bigr) \ =\ -\infty\,.
 \end{equation}
 For then, as $0<\beta_\ba,\beta_\bb<1$, if $i$ is large enough, then one of $z_\ba$ or $z_\bb$
 is less than $\epsilon$, which implies that $d(y,z)>\epsilon$ and 
 thus  $d(y,X_{\calA,w_i})>\epsilon$.

To establish~\eqref{Eq:quantity}, consider the linear function $\varphi$ defined for
$\lambda\in\R^\calA$ by 
\[
   \defcolor{\varphi(\lambda)}\ :=\ \beta_\ba\lambda(\ba)+\beta_\bb\lambda(\bb)-
   \sum_{\bg\in\calG}\alpha_\bg\lambda(\bg)\,.
\]
Then the limit~\eqref{Eq:quantity} is equivalent to 
 \begin{equation}\label{Eq:limit_repeated}
   \lim_{i\to\infty} \varphi(v_i)\ =\ -\infty\,.
 \end{equation}
 By the inequality~\eqref{Eq:liftIneq}, $\varphi(\lambda)<0$ for 
 $\lambda$ in the relative interior of $\tau$, and thus $\varphi$ is nonpositive on $\tau$.
 As $\calG$ is a subset of a facet $\calF$ of $\calS_\sigma$, the 
 inequality~\eqref{Eq:liftIneq} shows that $\varphi(\lambda)$ is also negative for $\lambda$ in
 the relative interior of $\sigma$.  

 If $\varphi(v_i)$ does not have the limit~\eqref{Eq:limit_repeated}, then there is an 
 $M$ with $M<\varphi(v_i)$ for infinitely many $v_i$.
 Then $M$ is negative, as $\varphi(v_i)<0$ for all but finitely many $i\in\N$ since all but
 finitely many $v_i$ lie in the relative interior of $\tau$.
 Since $\varphi$ is negative on the relative interior of the cone $\sigma$, there is some
 $v\in\sigma$ with $\varphi(v)=M$, and so there are infinitely many $v_i$ with
 $\varphi(v)<\varphi(v_i)$.
 Such a $v_i$ has $0<\varphi(v_i-v)$, which implies that $v_i-v\not\in\tau$.
 Consequently, infinitely many elements of $\{v_i-v\mid i\in\N\}$ do not lie in $\tau$, which
 contradicts Lemma~\ref{L:cone_bounded}. 
 This establishes the limit~\eqref{Eq:quantity} and shows that there is some $N\in\N$ such
 that if $i>N$ then $d(y,X_{\calA,w_i})>\epsilon$.

 Suppose that we are in the second case of $y_\bc\neq 0$ with $\bc$ not lying in any face
 of $\calS$.
 As $\sigma$ is a face of $\tau$, $\calS=\calS_\sigma$ is refined by $\calS_\tau$ and we see
 that $\bc$ does not lie in any face of $\calS_\tau$.
 By Lemma~\ref{L:convex_combinations} there is a relation~\eqref{Eq:convex_i} expressing $\bc$
 as a convex combination of the points of a facet $\calG$ of $\calS_\tau$.
 By Proposition~\ref{P:equations} this gives the valid equation on points $z\in X_{\calA,w_i}$, 
\[
    z_\bc\ =\ \frac{w_i(\bc)} {\prod_{\bg\in\calG} w_i(\bg)^{\gamma_{\bg}}}
    \cdot \prod_{\bg\in\calG} z_\bg^{\gamma_{\bg}}
   \qquad\mbox{with}\quad 0\leq\gamma_\bg\leq 1\,.
\]
 As $0\leq z_\bg\leq 1$ and $w_i=\Exp(v_i)$, this implies that 
\[
   z_\bc\ \leq\ \exp\bigl( v_i(\bc)-\sum_{\bg\in\calG}\gamma_\bg v_i(\bg)\bigr)\,.
\]
 We complete the proof by showing that 
\[
   \lim_{i\to\infty}
     \bigl( v_i(\bc)-\sum_{\bg\in\calG}\gamma_\bg v_i(\bg)\bigr)
    \ =\ -\infty\,.
\]
 Set 
 $\varphi(\lambda):=\lambda(\bc)-\sum_{\bg\in\calG}\gamma_{\bg}\lambda(\bg)$, for
 $\lambda\in\R^\calA$. 
 This limit becomes $\lim_{i\to\infty}\varphi(v_i)=-\infty$, which is proved by the same
 arguments as for the limit~\eqref{Eq:limit_repeated}.
 Thus in this second case there is a number $N\in\N$ such that if $i>N$, then 
 $d(y,X_{\calA,w_i})>\epsilon$.
\end{proof}

\begin{lemma}\label{L:second_third}
 If $y\in|\calS|$ is an accumulation point of 
 $\{X_{\calA,w_i}\mid i\in\N\}$, then $y\in X(\calS,w)$.
\end{lemma}

\begin{proof}
 Let $y\in|\calS|$, so that $y\in\simplex^\calF$ for some face $\calF$ of $\calS$.
 If $y$ is also an accumulation point of $\{X_{\calA,w_i}\mid i\in\N\}$, 
 then for all $\epsilon>0$ and for all $N>0$ there is an $i>N$ and  
 point $z\in X_{\calA,w_i}$ with $d(y,z)<\frac{1}{3}\epsilon$.
 Since $y_\ba=0$ for $\ba\in\calA\smallsetminus\calF$, we must have
\[
   \sum_{\ba\in\calA\smallsetminus\calF} z_\ba\ <\ \frac{1}{3}\epsilon\ ,
\]
 and so by Lemma~\ref{L:ell_1}, $d(z,\pi_{\calF}(z))<\frac{2}{3}\epsilon$, which implies that 
 $d(y,\pi_{\calF}(z))<\epsilon$.
 As $\pi_{\calF}(z)\in\pi_\calF(X_{\calA,w_i})=X_{\calF,w_i}$, this shows that $y$ is an
 accumulation point of $\{X_{\calF,w_i}\mid i\in\N\}$.
 Since we have $\lim_{i\to\infty}X_{\calF,w_i}=X_{\calF,w}$~\eqref{Eq:limit_F} and
 Lemma~\ref{L:complex_limit},
 we have $y\in X_{\calF,w}$.
\end{proof}

\begin{lemma}\label{L:third_third}
 Every point of $X(\calS,w)$ is a limit point of the sequence 
 $\{X_{\calA,w_i}\mid i\in\N\}$.
\end{lemma}

\begin{proof}
 We prove that every $x\in X^\circ_{\calF,w}$ for $\calF$ a facet of $\calS$ is a
 limit point of $\{X_{\calA,w_i}\}$.
 This suffices, as the union of these sets,
\[
   X^\circ(\calS,w)\ :=\ 
   \bigsqcup_{\calF\mbox{\scriptsize\ a facet of }\calS} X^\circ_{\calF,w}\,.
\]
 is a dense subset of $X(\calS,w)$.

 Let $\calF$ be a facet of $\calS$.
 For $\delta>0$, define 
\[
   \defcolor{B_\delta}\ :=\ \{y\in\simplex^\calF\mid y_\fb\geq\delta\mbox{ for }\fb\in\calF\}\,.
\]
 Let $x\in X^\circ_{\calF,w}$ and $\epsilon>0$.
 Since $x_\fb\neq 0$ for $\fb\in\calF$, there is a 
 $\delta>0$ with $x_\fb\geq 2\delta$ for $\fb\in\calF$.
 By Lemma~\ref{L:last} below, there is a number 
 $N_1$ such that if $i>N_1$ and $y\in B_{\delta}\cap X_{\calF,w_i}$, 
 there is a  point $z\in X_{\calA,w_i}$ with $d(z,y)<\epsilon$.
 We showed (in~\eqref{Eq:limit_F}) that 
\[
    \lim_{i\to\infty} X_{\calF,w_i}\ =\ X_{\calF,w}\,.
\]
 Thus there is a number $N\geq N_1$ such that if $i>N$, there is a 
 point $y\in X_{\calF,w_i}$ with $d(x,y)<\min\{\epsilon,\delta\}$.
Since  $|x_\fb-y_\fb|<\delta$ and $x_\fb\ge 2\delta$, we have $y_\fb>\delta$ for all
   $\fb\in\calF$, and thus $y\in B_{\delta}$.
 As $i>N_1$, there is a point $z\in X_{\calA,w_i}$ with $d(y,z)<\epsilon$.
 Therefore $d(x,z)\leq 2\epsilon$, which shows that $x$ is a limit point of 
 $\{X_{\calA,w_i}\}$.
\end{proof}

\begin{lemma}\label{L:last}
 Let $\calF$ be a facet of $\calS$ and $\delta,\epsilon>0$.
 Then there exists a number $N$ such that for every $i>N$ and 
 $y\in B_\delta\cap X_{\calF,w_i}$ the point $z\in X_{\calA,w_i}$ with
 $\pi_\calF(z)=y$ satisfies $d(y,z)<\epsilon$. 
\end{lemma}

\begin{proof}
 Let $\bd\in\calA\smallsetminus\calF$.
 As $\calF$ is a facet of $\calS_\sigma$ and $\calS_\sigma$ is refined by
 $\calS_\tau$, there is a facet $\calG$ of $\calS_\tau$ with $\calG\subset\calF$.
 By Lemma~\ref{L:convex_combinations} there is a relation~\eqref{Eq:affine_bd} 
 expressing $\bd$ as an affine combination of points of $\calG$, and by
 Proposition~\ref{P:equations} this gives the valid equation on points 
 $x\in X^\circ_{\calA,w_i}$,
 \begin{equation}\label{Eq:last_equation}
   x_\bd\ =\ \frac{w_i(\bd)}{\prod_{\bg\in\calG}w_i(\bg)^{\alpha_\bg}}
              \prod_{\bg\in\calG} x_{\bg}^{\alpha_\bg}\,.
 \end{equation}
 For each $\bd\in\calA\smallsetminus\calF$, fix one such affine
 expression~\eqref{Eq:affine_bd} for $\bd$ in terms of a subset $\calG$ of $\calF$ that is a
 facet in $\calS_\tau$, together with the corresponding equation~\eqref{Eq:last_equation} on 
 $X^\circ_{\calA,w_i}$. 

 For $y\in B_\delta\cap X_{\calF,w_i}$, we have $y\in X^\circ_{\calF,w_i}$, so there is a
 unique  $z\in X^\circ_{\calA,w_i}$ with $\pi_\calF(z)=y$.
 We find $z$ by first computing the number $y_\bd$ satisfying~\eqref{Eq:last_equation}
 (with $y_\bg$ substituted for $x_\bg$) for each $\bd\in\calA\smallsetminus\calF$. 
 Then the point $y'$ whose coordinates for $\fb\in\calF$ equal those of $y$ and whose
 other coordinates are these $y_\bd$ satisfies the equations~\eqref{Eq:last_equation}
 for $X^\circ_{\calA,w_i}$, but it is not a point of the standard simplex, $\simplex^\calA$ for
 the sum of its coordinates exceeds 1.
 Dividing each coordinate of $y'$ by this sum gives the point $z\in X^\circ_{\calA,w_i}$ 
 lying in the simplex $\simplex^\calA$ with $\pi_\calF(z)=y$.

 We extract from this discussion that the coordinate $z_\bd$ of $z$ is smaller than the
 coordinate $y_\bd$ of $y'$ that we computed from $y\in B_\delta\cap X_{\calF,w_i}$
 and~\eqref{Eq:last_equation}. 

 We show below that for every $\epsilon>0$ there is a number $N$ such that
 if $i>N$ and $y\in B_\delta\cap X_{\calF,w_i}$, then for each
 $\bd\in\calA\smallsetminus\calF$ the number $y_\bd$ that we compute from $y$
 and~\eqref{Eq:last_equation} is at most 
 $\frac{1}{2|\calA\smallsetminus\calF|}\epsilon$.
 Then if $z\in X_{\calA,w_i}$ is the point which projects to $y$, we have 
\[
   d(y,z)\ =\ 2\sum_{\bd\in\calA\smallsetminus\calF} z_\bd
    \ <\ 2\sum_{\bd\in\calA\smallsetminus\calF} y_\bd
    \ <\ 2\sum_{\bd\in\calA\smallsetminus\calF} 
          \frac{\epsilon}{2|\calA\smallsetminus\calF|}
    \ =\ \epsilon\,,
\]
 which will complete the proof.
 (The formula for $d(y,z)$ is from Lemma~\ref{L:ell_1}).

 First fix $\bd\in\calA\smallsetminus\calF$.
 For $y\in B_\delta$, the monomial from~\eqref{Eq:last_equation},
\[
   \prod_{\bg\in\calG}y_\bg^{\alpha_\bg}\,,    
\]
 is defined (as $y_\bg\geq\delta$) and is thus bounded on the compact set
 $B_\delta$ by some number, $L$.
 Then 
\[
   y_\bd\ \leq\ 
     \frac{w_i(\bd)}{\prod_{\bg\in\calG}w_i(\bg)^{\alpha_\bg}}\cdot L\,,     
\]
 if $y\in B_\delta\cap X_{\calF,w_i}$.
 We will show that the coefficient of $L$ has limit zero as $i\to\infty$.
 Since this holds for all $\bd\in\calA\smallsetminus\calF$, there is a number
 $N$ such that if $i>N$, then every number
 $y_\bd$ is bounded by $\frac{1}{2|\calA\smallsetminus\calF|}\epsilon$, which will complete the
 proof. 

 Taking logarithms, this limit being zero is equivalent to 
 \[
   \lim_{i\to\infty} 
    \bigl( v_i(\bd)\ -\ \sum_{\bg\in\calG} \alpha_\bg v_i(\bg)\bigr)
    \ =\ -\infty\,.       
 \]
 Define the linear function $\varphi$ on $\R^\calA$ by 
\[
  \defcolor{\varphi(\lambda)}\ :=\ 
     \lambda(\bd)\ -\ \sum_{\bg\in\calG} \alpha_\bg \lambda(\bg)\,,   
\]   
 where $\lambda\in\R^\calA$.
 Then our limit becomes $\lim_{i\to\infty}\varphi(v_i)=-\infty$, which is proved by the same
 arguments as for the limit~\eqref{Eq:limit_repeated}.
\end{proof}

\section*{Acknowledgements} 
The first and third authors would like to thank Ragni Piene for many useful conversations, and we all thank the
referee for constructive comments.


\def\cprime{$'$}
\providecommand{\bysame}{\leavevmode\hbox to3em{\hrulefill}\thinspace}
\providecommand{\MR}{\relax\ifhmode\unskip\space\fi MR }
\providecommand{\MRhref}[2]{%
  \href{http://www.ams.org/mathscinet-getitem?mr=#1}{#2}
}
\providecommand{\href}[2]{#2}


\end{document}